\documentclass{amsart}

\theoremstyle{theorem}
\newtheorem{theorem}{Theorem}

\newtheorem{lemma}{Lemma}

\DeclareMathOperator{\card}{\text{card}}

\DeclareMathOperator{\qqand}{\qquad\text{and}\qquad}

\newcommand{\N}{\ensuremath{ \mathbf N }}
\newcommand{\Z}{\ensuremath{\mathbf Z}}

\newcommand{\beq}{\begin{equation}}
\newcommand{\eeq}{\end{equation}}
\newcommand{\benum}{\begin{enumerate}}
\newcommand{\eenum}{\end{enumerate}}

\usepackage{amsmath,amssymb,amsthm}

\title{Sums of Finite Sets of Integers, II}
\author{Melvyn B. Nathanson}

\address{Lehman College (CUNY), Bronx, NY 10468}
\email{melvyn.nathanson@lehman.cuny.edu}



\begin{document}
\maketitle

\begin{abstract}
A fundamental result in additive number theory  states that, 
for every finite set $A$ of integers, the $h$-fold sumset $hA$ 
has a very simple and beautiful structure for all sufficiently large $h$.  
Let $(hA)^{(t)}$ be the set of all integers in the sumset  $hA$ that have at least 
$t$ representations as a sum of $h$ elements of $A$.  
It is proved that the set $(hA)^{(t)}$ has a similar structure.  
\end{abstract}

\section{Structure of sumsets.}
G.  H. Hardy and E. M. Wright~\cite[p. 361]{hard-wrig08} clearly stated 
the general problem of additive number theory.  
\begin{quotation}

\noindent
Suppose that $A$ or 
$
a_1, a_2,a_3,\ldots
$
 is a given system of integers.  Thus $A$ might contain all 
the positive integers, or the squares, or the primes.  We consider all representations of 
an arbitrary positive integer $n$ in the form 
\[
n = a_{i_1} + a_{i_2} + \cdots + a_{i_s},
\]
\ldots. We denote by $r(n)$ the number of  such representations.  Then what can we say about $r(n)$?
\end{quotation}
Many classical problems are still unsolved.  For example, we do not know what numbers 
are sums of four cubes.   

Much recent work concerns sums of arbitrary sets of integers.  
The \textit{$h$-fold sumset} of a set $A$ of integers is the set $hA$ consisting 
of all integers that can be represented as the sum 
of $h$ not necessarily distinct elements of $A$.  Additive number theory studies $h$-fold sumsets.  
For every finite or infinite set $A$ of integers, we would like to know the structure 
of the sumsets $hA$ for small $h$ 
and, asymptotically, as $h$ goes to infinity.  
A fundamental theorem of additive number theory, published 50 years ago 
in~\cite{nath1972-7,nath1996bb}, explicitly   
solves the asymptotic problem for finite sets of integers.  

Define the \textit{interval of integers} 
$[u,v] = \{n \in \Z: u \leq n \leq v\}$.
For every set $D$ and integer $w$, let $w-D = \{w-d:d\in D\}$.

\begin{theorem}           \label{FiniteSets:theorem:MBN-1}
Let $A = \{a_0,a_1,\ldots, a_k\} $ be a finite set of integers such that    
\[
0 = a_0 < a_1 < \cdots < a_k 
\qqand
\gcd(A) = 1.
\]
Let 
\[
h_1 = (k-1)(a_k-1)a_k + 1.
\]
There are nonnegative integers $c_1$ and $d_1$ 
and finite sets $C_1$ and $D_1$ with 
\[
C_1 \subseteq  [0,c_1 -2]  \qqand D_1 \subseteq [0,d_1 -2]
\]
such that 
\[
hA = C_1  \cup [c_1,ha_k -d_1] \cup \left(ha_k  - D_1 \right)  
\]
for all $h \geq h_1$.  
\end{theorem}

Thus, for $h$ sufficiently large, the sumset $hA$ consists of a long interval of consecutive 
integers, with a small initial fringe $C_1$ and small terminal fringe $ha_k-D_1$.  
This structure is rigid.  In the sumset $(h+1)A$, the  length of the interval increases by $a_k$, 
the initial fringe $C_1$ is unchanged, and the terminal fringe translates to the right by $a_k$.  

The integer $FN_1(A) = c_1-1$ is the \textit{Frobenius number} of the set $A$, 
that is, the largest number that cannot be represented as a nonnegative integral
linear combination of elements of $A$. 
This is  often presented as the \textit{Frobenius coin problem}:  Find the largest amount that cannot 
be obtained using only coins with denominations $a_1,\ldots, a_k$.

Smaller values for the number $h_1$ have been obtained 
by Wu, Chen, and Chen~\cite{wu-chen-chen11}, 
Granville and   Shakan~\cite{gran-shak20}, 
and Granville and   Walker~\cite{gran-walk20}. 

Let $B$ be a finite set of integers with $|B| \geq 2$.
If $\min(B) = b_0$ and $\gcd(B-b_0) = d$, then the ``normalized set'' 
\[
A = \left\{ \frac{b-b_0}{d}: b \in B \right\} 
\]
is a finite set of nonnegative integers with $\min(A) = 0$ and $\gcd(A) = 1$.  
We have 
\[
hB =  hb_0 + \{dx: x \in hA\}.
\]
for all positive integers $h$.  
Thus, Theorem~\ref{FiniteSets:theorem:MBN-1} describes the asymptotic structure 
of the sumsets of every finite set of integers.  

Han, Kirfel, and Nathanson~\cite{nath1998-92} 
extended Theorem~\ref{FiniteSets:theorem:MBN-1} to linear forms 
of finite sets of integers.
Khovanski{\u\i}~\cite{khov92,khov95} and Nathanson~\cite{nath2000-98} 
proved the exact polynomial growth of sums of finite sets of lattice points, 
and, more generally, of linear forms of finite subsets of any additive abelian semigroup.

\section{Representation functions.}
Let $A$ be a set of integers.  
For every positive integer $h$, the \textit{$h$-fold representation function} 
$r_{A,h}(n)$ counts the number of representations of $n$ 
as the sum of $h$ elements of $A$.  Thus, 
\[
r_{A,h}(n) 
= \card\left\{ (a_{j_1},\ldots, a_{j_h}) \in A^h : n = \sum_{i=1}^h a_{j_i} 
\text{ and }  a_{j_1} \leq \cdots \leq a_{j_h}  \right\}.
\]
Equivalently, if $\N_0^{A}$ is the set of all sequences of nonnegative integers indexed 
by the elements of  $A$, then 
\[
r_{A,h}(n)  =\card \left\{ (u_a)_{a\in A} \in \N_0^{A}:   \sum_{a\in A} u_a a = n \text{ and } 
\sum_{a\in A} u_a = h\right\}.
\]

For every positive integer $t$, let $(hA)^{(t)}$ be the set of all integers $n$ 
that have at least $t$ representations as the sum of $h$ elements of $A$, that is, 
\[
(hA)^{(t)} = \{n \in \Z : r_{A,h}(n) \geq t \}.
\]
The following result completely determines the structure of the sumsets $(hA)^{(t)}$ for all $t$
and for all sufficiently large $h$.

\begin{theorem}           \label{FiniteSets:theorem:MBN-2}
Let $k \geq 2$, and let $A = \{a_0,a_1,\ldots, a_k\} $ be a finite set of integers such that    
\[
0 = a_0 < a_1 < \cdots < a_k 
\qqand
\gcd(A) = 1.
\]
For every positive integer $t$,  let 
\[
h_t = (k-1)(ta_k-1)a_k + 1.
\]
There are nonnegative integers $c_t$ and $d_t$ 
and finite sets $C_t$ and $D_t$ with 
\[
C_t \subseteq  [0,c_t -2]  \qqand D_t \subseteq [0,d_t -2]
\]
such that 
\[
(hA)^{(t)} = C_t  \cup [c_t,ha_k -d_t] \cup \left(ha_k  - D_t \right) 
\]
 for all $h \geq h_t$.
\end{theorem}
  
It is remarkable that the sumsets $(hA)^{(t)}$ have the same structure as the sumset $hA$.

\section{Proof of Theorem~\ref{FiniteSets:theorem:MBN-2}.}    \label{FiniteSets:section:TheoremProof}
The proofs of the following lemmas are in Section~\ref{FiniteSets:section:LemmaProofs}.  

\begin{lemma}                         \label{FiniteSets:lemma:inclusion}
 Let A\ be a set of integers.  
For all positive integers $h$ and $t$, 
\[
(hA)^{(t)} + A \subseteq    \left( (h+1)A\right)^{(t)}. 
\]
\end{lemma}

\begin{lemma}                          \label{FiniteSets:lemma:CD-1}
Let $k \geq 2$ and let $A = \{a_0,a_1,\ldots, a_k \}$ be a  finite set of integers  with   
\[
0 = a_0 < a_1 < \dots < a_{k} \qqand \gcd(A) = 1. 
\] 
For every positive integer $t$, let 
\beq                                                    \label{FiniteSets:Ct}
c'_t  = (ta_{k}-1)\sum_{j=1}^{k-1}  a_j
\eeq
and 
\beq                                                    \label{FiniteSets:Dt}
d'_t =   (k-1) (t a_{k} -1) a_{k}.
\eeq
For every positive integer $h$, 
\beq                                                    \label{FiniteSets:CDinterval-0}
[c'_t, h a_{k}  - d'_t] \subseteq (h A)^{(t)}.
\eeq
\end{lemma}

We now prove Theorem~\ref{FiniteSets:theorem:MBN-2}.

\begin{proof}
Let $t$ be a positive integer.  Define $c'_t$ by~\eqref{FiniteSets:Ct} and $d'_t$ by~\eqref{FiniteSets:Dt}. 
By Lemma~\ref{FiniteSets:lemma:CD-1}, the set $ (h_t A)^{(t)}$ contains 
the  interval $[c'_t, h_t a_{k}  - d'_t] $.
Let $c_t$ and $d_t$ be the smallest integers such that 
\[
[c'_t, h_t a_{k}  - d'_t] \subseteq [c_t, h_t a_{k}  - d_t] \subseteq (h_t A)^{(t)}.
\]
Thus, $c_t \leq c'_t$ and $d_t \leq d'_t$.  It follows that 
\[
c_t-1 \notin  (h_t A)^{(t)} \qqand h_t a_{k}  - d_t +1 \notin  (h_t A)^{(t)}.
\]
Define  the finite sets $C_t$ and $D_t$ by 
\[
C_t = [0, c_t-1] \cap  (h_t A)^{(t)} 
\]
and 
\[
h_t a_k - D_t =  [ h_t a_k - d_t + 1, h_t a_k] \cap  (h_t A)^{(t)}.  
\]
This gives 
\[
(h_t A)^{(t)}= C_t \cup [c_t, h_t a_k - d_t] \cup (h_t a_k-D_t).
\]
We shall prove that 
\beq                             \label{FiniteSets:induction}
(h A)^{(t)} = C_t \cup [c_t, h a_k - d_t] \cup (h a_k-D_t)
\eeq
for all $h \geq h_t$.

The proof is by induction on $h$.  
Assume that~\eqref{FiniteSets:induction} is true for some $h \geq h_t$.
Because $\{0,a_k\} \subseteq A$,  Lemma~\ref{FiniteSets:lemma:inclusion} gives  
\beq                                                                                                \label{FiniteSets:hAinclusion}
(hA)^{(t)}  \cup \left((hA)^{(t)} + a_k \right)  \subseteq  (hA)^{(t)} + A \subseteq    \left( (h+1)A\right)^{(t)} 
\eeq
and so 
\[
C_t \subseteq (h A)^{(t)}  \subseteq \left( (h+1)A\right)^{(t)}. 
\]
Because $c'_t \leq d'_t = h_t -1 \leq h-1$ and $a_k \geq 2$, we have 
\[
c_t + d_t  \leq c'_t  + d'_t \leq 2d'_t \leq a_k(h_t-1) \leq a_k(h-1). 
\]
Therefore, 
\[
c_t + a_k  \leq ha_k -d_t 
\]
and
\[
 [c_t, c_t + a_k] \subseteq  [c_t, h a_k - d_t] \subseteq  (h A)^{(t)} 
 \subseteq \left( (h+1)A\right)^{(t)}.  
\]
By~\eqref{FiniteSets:hAinclusion},  
\begin{align*}
 [c_t+a_k, (h+1) a_k - d_t] & =  a_k + [c_t, h a_k - d_t] \\
 &  \subseteq a_k + (hA)^{(t)} \\
&  \subseteq ((h+1)A)^{(t)} 
\end{align*}
and
\begin{align*}
(h+1) a_k-D_t & =  a_k + \left( h a_k-D_t \right)  \\
& \subseteq  a_k + (hA)^{(t)}  \\
&  \subseteq ((h+1)A)^{(t)}.  
\end{align*}
Therefore,
\[
B^{(t)} = C_t \cup [c_t, (h+1)a_k - d_t] \cup ( (h+1) a_k-D_t) \subseteq  ((h+1)A)^{(t)}. 
\]
We must prove that $B^{(t)} =  ((h+1)A)^{(t)}$. 

We have $A \subseteq [0,a_k]$ and 
\[
((h+1)A)^{(t)} \subseteq  (h+1)A \subseteq (h+1)[0,a_k]  = [0,(h+1)a_k].  
\]
Thus, if $n \in ((h+1)A)^{(t)} \setminus B^{(t)}$, then $0  \leq n \leq c_t -1$ 
or $(h+1) a_k - d_t + 1 \leq n \leq (h+1) a_k$.

If $n \in ((h+1)A)^{(t)} \setminus B^{(t)}$ and $n \leq c_t -1$, 
then 
\[
n \notin C_t = [0, c_t-1] \cap  (hA)^{(t)} 
\]  
and so $r_{A,h}(n) \leq t-1$.
However, $n \in ((h+1)A)^{(t)}$ means $r_{A,h+1}(n) \geq t$.  
Therefore, $n$ has at least $t$ representations as the sum of $h+1$ elements of A, 
but at most $t-1$ representations as the sum of $h$ elements of A.  
It follows that $n$ has at least one representation 
as the sum of $h+1$ positive elements of A, and so
\[
n \leq c_t -1\leq c'_t -1 \leq h_t  \leq  h < (h+1)a_1 \leq n
\]
which is absurd.  Therefore, if $n \in ((h+1)A)^{(t)}$ and $n < c_t$, 
then $n \in C_t \subseteq B^{(t)}$.

If $n \in ((h+1)A)^{(t)} \setminus B^{(t)}$ and $n \geq (h+1) a_k - d_t + 1$, then  
\[
n \notin (h+1)a_k - D_t
\]
and so 
\[
n-a_k \notin  ha_k - D_t = [ha_k- d_t+1, ha_k] \cap (hA)^{(t)}.
\]
Therefore, $r_{A,h}(n - a_k) \leq t-1$.  
However, $n \in ((h+1)A)^{(t)}$ implies that $r_{A,h+1}(n) \geq t$,  
and so there is at least one representation of $n = a_{i_1} + \cdots + a_{i_{h+1}}$ 
with  $a_{i_j} \leq a_{k-1}$ for all $j \in [1,h+1]$.    
It follows  that 
\[
(h+1)a_k - d_t +1 \leq n \leq (h+1) a_{k-1} \leq (h+1) (a_k-1)
\]
and so 
\[
h_t \leq h \leq d_t - 2 \leq d'_t -2 =  h_t -3 
\]
which is absurd.  
Therefore, 
\[
n \in (h+1)a_k - D_t \subseteq B^{(t)}.
\]
It follows that $(h+1)A)^{(t)} = B^{(t)}$.  
This completes the proof.  
\end{proof}

If $A$ is a finite set of integers with $\min(A) = 0$ and $\gcd(A) = 1$,   
then $FN_t(A) = c_t-1$ is the largest integer that does not have $t$
representations as the sum of elements of $A$.  
Equivalently,  $r_{A,h}(c_t-1) < t$ for all $h \geq 1$.  
We have the increasing sequence 
\[
FN_1(A)  \leq \cdots \leq FN_t(A)  \leq FN_{t+1}(A) \leq \cdots. 
\]
There is no efficient algorithm to compute the numbers $FN_t(A)$, 
and very little is known about them.

\section{Symmetry.}
 Let $A = \{a_0,a_1,\ldots, a_k \}$ be a finite set of integers  with 
\[
0 = a_0 < a_1 < \cdots < a_k.
\]
The \textit{dual set} 
\[
A^* =  \max(A) - A = \{ a_k - a_j: j \in [0,k] \} 
\]
satisfies $\left(A^*\right)^* = A$ and $\gcd(A) = \gcd\left( A^* \right)$. 
Because $ha_k = \max(hA) = \max(hA^*)$, we have 
\[
n = \sum_{j=1}^h a_{i_j} \in hA
\]
if and only if 
\[
ha_k - n = \sum_{j=1}^h \left( a_k- a_{i_j} \right) \in hA^*.
\]
Thus, $(hA)^* = hA^*$.  Similarly, 
\[
 \left( \left(hA\right)^{(t)} \right)^* =  \left(hA^*\right)^{(t)}   
\]
for all positive integers $h$ and $t$.  It follows that if 
\[
(hA)^{(t)}  = C_t  \cup [c_t,ha_{k} -d_t] \cup \left(ha_{k}  - D_t \right), 
\] 
then 
\begin{align*}
\left(hA^*\right)^{(t)} & = \left( \left(hA\right)^{(t)} \right)^*  \\
& = \left(  C_t  \cup [c_t,ha_{k} - d_t] \cup \left(ha_{k}  - D_t \right)\right)^* \\
& =  D_t  \cup [d_t,ha_{k} - c_t] \cup \left(ha_{k}  - C_t \right). 
\end{align*}
If $A = A^*$, then $c_t = d_t$ and $C_t = D_t$.

\section{Proofs of the lemmas.}                 \label{FiniteSets:section:LemmaProofs} 
Now we prove Lemmas~\ref{FiniteSets:lemma:inclusion} and~\ref{FiniteSets:lemma:CD-1} 
from Section~\ref{FiniteSets:section:TheoremProof}.

\begin{proof}[Proof of  Lemma~\ref{FiniteSets:lemma:inclusion}.]
Let $n \in (hA)^{(t)}$.  
Because  $r_{A,h}(n) \geq t$,  for $s \in [1,t]$ 
there are distinct sequences $(u_{a,s})_{a\in A}$ of nonnegative integers 
that  satisfy 
\[
\sum_{a\in  A} u_{a,s}  a = n 
\qqand
\sum_{a\in  A } u_{a,s} = h. 
\]
For all $a, a' \in A$, let 
\[
u'_{a,s} = \begin{cases}
u_{a,s} & \text{if $a\neq a'$} \\
u_{a,s} + 1 & \text{if $a = a'$.} 
\end{cases}
\]
The sequences $(u'_{a,s})_{a\in A}$ are also distinct for $s \in [1,t]$, and satisfy 
\[
\sum_{a\in  A } u'_{a,s}  a = n + a' 
\qqand
\sum_{a\in  A} u'_{a,s} =  h + 1. 
\]
It follows that $r_{A,h+1}(n+a') \geq t$, and so $(hA)^{(t)} + a' \subseteq ((h+1)A)^{(t)}$ 
for all $a' \in A$.  This completes the proof.  
\end{proof}

\begin{proof}[Proof of Lemma~\ref{FiniteSets:lemma:CD-1}.]

If $ha_k < c'_t + d'_t$, then the interval $[c'_t, h a_{k}  - d'_t]$ is empty 
and~\eqref{FiniteSets:CDinterval-0} is true.

Let $ha_k \geq c'_t + d'_t$ and
\[
n \in [c'_t, h a_{k}  - d'_t].  
\]
Because $\gcd(A) = \gcd(a_1,\ldots, a_{k}) = 1$,  there exist 
integers $x'_1,\ldots, x'_{k}$ such that 
\[
n = \sum_{j=1}^{k} x'_j a_j
\]
and so 
\[
n \equiv \sum_{j=1}^{k-1} x'_j a_j \pmod{a_{k}}. 
\]
For all integers $s$, the interval $[(s-1)a_{k}, sa_{k}-1]$ is a complete set of 
representatives for the congruence classes modulo $a_{k}$.  
It follows that, for all $j \in [1, k-1]$ and $s \in [1,t]$, there exist unique integers 
\beq                                                    \label{FiniteSets:CD-cong}
x_{j,s} \in [(s-1)a_{k}, sa_{k}-1] 
\eeq
such that 
\[
x'_j \equiv x_{j,s} \pmod{a_{k}}.
\]
Therefore,  
\[
n \equiv \sum_{j=1}^{k-1} x_{j,s} a_j \pmod{a_{k}}.
\]
There is a unique integer $x_{k,s}$ such that 
\beq                                                    \label{FiniteSets:CD}
n = \sum_{j=1}^{k} x_{j,s} a_j .
\eeq
The inequality
\[
\sum_{j=1}^{k-1} x_{j,s} a_j \leq  \sum_{j=1}^{k-1} (sa_{k}-1) a_j 
\leq  (ta_{k}-1)\sum_{j=1}^{k-1}  a_j = c'_t \leq n 
\]
implies   
\[
x_{k,s}a_{k} = n - \sum_{j=1}^{k-1} x_{j,s} a_j \geq 0.  
\]
Thus, $x_{k,s} \geq 0$ for all  $s \in [1,t]$, 
and so~\eqref{FiniteSets:CD} is a nonnegative 
integral linear combination of elements of $A$. 

We have    
\[
x_{k,s}a_{k} \leq  n \leq h a_{k}  - d'_t = h a_{k}  - (k-1)(t a_{k} -1)a_{k} 
\]
and so 
\[
x_{k,s} \leq h  - (k-1)(t a_{k} -1). 
\]
Therefore, 
\begin{align*}
\sum_{i=1}^{k} x_{i,s} & = \sum_{i=1}^{k-1} x_{i,s} + x_{k,s} \\
&  \leq (k-1)(s a_{k} -1) +  h  - (k-1)(t a_{k} -1) \\
& = h - (k-1)(t-s)a_k \\ 
& \leq h
\end{align*}
and $n \in hA$.  
It follows from~\eqref{FiniteSets:CD-cong} that, for $s \in [1,t]$,  the $k$-tuples 
\[
\left( x_{1,s}, x_{2,s}, \ldots, x_{k-1,s}, x_{k,s} \right)
\]
are distinct, and so the representations~\eqref{FiniteSets:CD} 
are distinct.  Therefore, $r_{A,h}(n) \geq t$ and 
\[
[c'_t, h a_{k}  - d'_t] \subseteq (h A)^{(t)}.
\]
This proves Lemma~\ref{FiniteSets:lemma:CD-1}. 
\end{proof}

\def\cprime{$'$} \def\cprime{$'$} \def\cprime{$'$}
\providecommand{\bysame}{\leavevmode\hbox to3em{\hrulefill}\thinspace}
\providecommand{\MR}{\relax\ifhmode\unskip\space\fi MR }
\providecommand{\MRhref}[2]{%
  \href{http://www.ams.org/mathscinet-getitem?mr=#1}{#2}
}
\providecommand{\href}[2]{#2}

\end{document}